\documentclass{amsart}

\usepackage[colorlinks=true, urlcolor=blue, citecolor=blue, linkcolor=blue, hypertexnames=false, hyperfootnotes=true]{hyperref}
\usepackage{amssymb}
\usepackage{mathrsfs}
\usepackage{dsfont}
\usepackage{mathtools}
\usepackage{tikz}

\newcommand\barrow{\textstyle\mathop{\rightarrow}_{}^{\hspace{-8pt}\bullet}}
\newcommand\arrowb{\textstyle\mathop{\rightarrow}_{\hspace{-8pt}\bullet}^{}}

\newcommand\carrow{\textstyle\mathop{\rightarrow}_{}^{\hspace{-8pt}\circ}}
\newcommand\arrowc{\textstyle\mathop{\rightarrow}_{\hspace{-8pt}\circ}^{}}
\newcommand\carrowc{\textstyle\mathop{\rightarrow}_{\hspace{-8pt}\circ}^{\hspace{-8pt}\circ}}
\newcommand\carrowb{\textstyle\mathop{\rightarrow}_{\hspace{-8pt}\bullet}^{\hspace{-8pt}\circ}}

\newtheorem{thm}{Theorem}

\newtheorem{prp}{Proposition}

\theoremstyle{definition}

\newtheorem{dfn}{Definition}

\numberwithin{equation}{section}

\author{Tristan Bice}
\address{Federal University of Bahia\\
Salvador\\
Brazil}
\email{Tristan.Bice@gmail.com}
\thanks{This research has been supported by an IMPA (Brazil) postdoctoral fellowship.}

\title{Semicontinuity in Ordered Banach Spaces}

\begin{document}

\begin{abstract}
We extend the C*-algebra semicontinuity theory of Akemann, Brown and Pedersen to (pre)ordered Banach spaces.
\end{abstract}

\maketitle

\section*{Motivation}

To understand the kind of results we wish to generalize, let us first recall some basic facts about semicontinuity.  By definition, a function $f$ from a topological space $Q$ to $\mathbb{R}$ is \emph{lower semicontinuous} (lsc) if, for all $q\in Q$ and all nets $(q_\lambda)\subseteq Q$,
\[q_\lambda\rightarrow q\quad\Rightarrow\quad f(q)\leq\lim\inf f(q_\lambda).\]
In more topological terms
\[f\text{ is lsc}\quad\Leftrightarrow\quad f\text{ is Scott continuous},\]
 where the the Scott topology of $\mathbb{R}$ consists of open sets of the form $(r,\infty)$, for all $r\in\mathbb{R}$.  For compact Hausdorff $Q$ we have, in more order theoretic terms,
\[f\text{ is lsc}\quad\Leftrightarrow\quad f_\lambda\rightarrow f\text{ pointwise, for an increasing net of continuous functions }(f_\lambda).\]
And by Dini's theorem, this convergence must be uniform iff $f$ is continuous, i.e.
\[f\text{ is continuous}\quad\Leftrightarrow\quad f\text{ is finite among lsc functions}.\]

In more general terms, what we have here is an ordered Banach space $X$ contained in a larger ordered Banach space $Y$ together with a set of positive functionals $Q$ on $Y$ considered in the weak topology induced by $X$, specifically
\begin{itemize}
\item $X=C(Q)=C(Q,\mathbb{R})=$ the continuous functions from $Q$ to $\mathbb{R}$.
\item $Y=B(Q)=B(Q,\mathbb{R})=$ the bounded functions from $Q$ to $\mathbb{R}$.
\item $Q\subseteq Y^*_+$, identifying points with their evaluation functionals.
\end{itemize}
For every $f\in Y$, we noted that $f|_Q$ is lsc iff $f|_Q$ is a pointwise limit of $f_\lambda|_Q$, for some increasing $(f_\lambda)\subseteq X$, where the convergence is necessarily uniform iff $f\in X$.

For a general ordered Banach space $X$ it is natural to take $Y=X^{**}$ and $Q=X^{*1}_+=$ the positive unit ball of $X^*$, which is compact Hausdorff in the weak* topology.  In this general situation, we want to know
\begin{enumerate}
\item Is $f\in X^{**}$ still lsc on $Q$ iff $f_\lambda\xrightarrow{\mathrm{w}^*}f$ for increasing $(f_\lambda)\subseteq X$?
\item Does $X$ still consist precisely of the finite lsc elements in $X^{**}$?
\end{enumerate}

The first question was investigated in detail for the self-adjoint part of a C*-algebra in \cite{AkemannPedersen1973} and \cite{Brown1988}, where a positive answer was given for both unital and separable C*-algebras.  The general case still appears to be open (see \cite{Brown2014}) although a positive answer was again obtained in \cite{AkemannPedersen1973} and \cite{Brown1988} by enlarging the set of weak* limits of increasing nets either to its norm closure or to limits of `almost' increasing nets.  Our first goal is to simplify and generalize these results to ordered Banach spaces using the non-symmetric distance theory from \cite{Bice2016b}.

This does not quite, however, generalize the original situation under consideration.  For if $X=C(Q)$ then $Q$ consists only of the non-zero extreme points of $X^{*1}_+$, and hence $B(Q)$ is only the `atomic part' of $X^{**}$.  But we can simultaneously generalize both these situations by replacing $\mathbb{R}$ with an ordered normed space $X$ and considering $C(Q,X)$ canonically embedded in $B(Q,X^{**})$.  This was also considered in \cite{Brown1988}, but only for the specific case $Q=\mathbb{N}\cup\{\infty\}$ and $X=\mathcal{K}(H)_\mathrm{sa}$.  Here again we will generalize to arbitrary ordered Banach spaces by considering an appropriate version of the Scott topology on lsc elements of $X^{**}$.

\section*{Outline}

In \autoref{D} we start with some general results for distances $\mathbf{d}$, i.e. functions merely satisfying the triangle inequality.  In particular, we generalize Dini's theorem in \autoref{Dini}, show that the $\mathbf{d}$-finite continuous functions are Yoneda complete in \autoref{Zcont} and characterize $\mathbf{d}$-algebraic distance spaces in \autoref{CauchyCauchy}.  We move on to preordered Banach spaces in \autoref{NS}, generalizing results from \cite{AkemannPedersen1973} and \cite{Brown1988} in \autoref{X**sc} and \autoref{CYX}.

\section{Distance Spaces}\label{D}

Even though our primary interest is in ordered normed spaces, it is more natural to do some preliminary work in more general non-symmetric distance spaces.

First define the composition $\mathbf{d}\circ\mathbf{e}$ of any $\mathbf{d},\mathbf{e}:X\times X\rightarrow[0,\infty]$ by
\begin{gather}
\mathbf{d}\circ\mathbf{e}(x,y)=\inf_{z\in X}\mathbf{d}(x,z)+\mathbf{e}(z,y).\nonumber\\
\textbf{From now on, we assume $\mathbf{d}$ a \emph{distance} on $X$ meaning}\nonumber\\
\mathbf{d}\leq\mathbf{d}\circ\mathbf{d}.\tag{$\triangle$}\label{tri}
\end{gather}
When then get a transitive relation $\leq^\mathbf{d}$ defined by
\[x\leq^\mathbf{d}y\quad\Leftrightarrow\quad\mathbf{d}(x,y)=0.\]
As in \cite{Goubault2013} Definition 6.1.1, we call $\mathbf{d}$ a \emph{hemimetric} if $\leq^\mathbf{d}$ is also reflexive, i.e. a preorder, and a \emph{quasimetric} if $\leq^\mathbf{d}$ is also antisymmetric, i.e. a partial order.

\subsection{Topology}

Just as with metrics, we can use $\mathbf{d}$ to define balls which generate a natural topology on $X$.  In normed spaces this corresponds to the usual norm topology, but we also need an analog of the weak* topology, and for this it turns out holes are more important.  We will also need an analog of the Scott topology, which is still generated by balls but only with centres in a specific subset of $X$.

So define the open upper/lower balls/holes with centre $c\in X$ and radius $\epsilon$ by
\begin{align*}
c^\bullet_\epsilon\ &=\ \{x\in X:\mathbf{d}(c,x)<\epsilon\}.\\
c_\bullet^\epsilon\ &=\ \{x\in X:\mathbf{d}(x,c)<\epsilon\}.\\
c^\circ_\epsilon\ &=\ \{x\in X:\mathbf{d}(x,c)>\epsilon\}.\\
c_\circ^\epsilon\ &=\ \{x\in X:\mathbf{d}(c,x)>\epsilon\}.
\end{align*}
For any $C\subseteq X$, let $C^\bullet$, $C_\bullet$, $C^\circ$, $C_\circ$, $C^\bullet_\bullet$, $C^\bullet_\circ$, $C^\circ_\bullet$ and $C^\circ_\circ$ denote the topologies on $X$ generated by the corresponding balls and holes with centres in $C$, i.e. by arbitrary unions of finite intersections.  
Denote convergence in these topologies by $\barrow$, $\arrowb$, $\carrow$, $\arrowc$, etc. so, for any net $(x_\lambda)\subseteq X$,
\begin{align*}
x_\lambda\barrow x\quad&\Leftrightarrow\quad\forall c\in C\ \limsup\mathbf{d}(c,x_\lambda)\leq\mathbf{d}(c,x).\\
x_\lambda\arrowb x\quad&\Leftrightarrow\quad\forall c\in C\ \limsup\mathbf{d}(x_\lambda,c)\leq\mathbf{d}(x,c).\\
x_\lambda\carrow x\quad&\Leftrightarrow\quad\forall c\in C\ \liminf\ \mathbf{d}(x_\lambda,c)\geq\mathbf{d}(x,c).\\
x_\lambda\arrowc x\quad&\Leftrightarrow\quad\forall c\in C\ \liminf\ \mathbf{d}(c,x_\lambda)\geq\mathbf{d}(c,x).
\end{align*}
Unless otherwise stated, we take $C=X$.
Also, just to be clear, by a net we mean a set indexed by a directed set $\Lambda$, i.e. we have (possibly non-reflexive) transitive $\mathbin{\prec}\subseteq\Lambda\times\Lambda$ satisfying $\forall\gamma,\delta\ \exists\lambda\ (\gamma,\delta\prec\lambda)$, with $\liminf$ and $\limsup$ defined by
\begin{align*}
\liminf_\lambda r_\lambda&=\lim_\gamma\inf_{\gamma\prec\lambda}r_\lambda.\\
\limsup_\lambda r_\lambda&=\lim_\gamma\sup_{\gamma\prec\lambda}r_\lambda.
\end{align*}

To see how hole topologies are analogous to product topologies, let us consider functions $X^Q$ from a set $Q$ to $X$ with respect to the supremum distance
\[\sup\!\text{-}\mathbf{d}(f,g)=\sup_{p\in Q}\mathbf{d}(f(p),g(p)).\]

\begin{prp}\label{X^Yprp}
\hfill$\forall p\in Q\ f_\lambda(p)\carrow f(p)\quad\Rightarrow\quad f_\lambda\carrow f$.\hfill\null\\
If $X$ has a $\leq^\mathbf{d}$-maximum then\hspace{7pt}$\forall p\in Q\ f_\lambda(p)\carrow f(p)\quad\Leftrightarrow\quad f_\lambda\carrow f$.
\end{prp}

\begin{proof}
Assume $f_\lambda(p)\carrow f(p)$, for all $p\in Q$.  For any $g\in X^Q$ and $r<\sup$-$\mathbf{d}(f,g)$, we have $p\in Q$ with $r<\mathbf{d}(f(p),g(p))\leq\liminf_\lambda\mathbf{d}(f_\lambda(p),g(p))\leq\liminf_\lambda\sup\!\text{-}\mathbf{d}(f_\lambda,g)$.  Thus $\sup$-$\mathbf{d}(f,g)\leq\liminf_\lambda\sup\!\text{-}\mathbf{d}(f_\lambda,g)$ and hence $f_\lambda\carrow f$, as $g$ was arbitrary.


Now assume $X$ has a $\leq^\mathbf{d}$-maximum $1$ and $f_\lambda\carrow f$.  For any $p\in Q$ and $x\in X$, define $g\in X^Q$ by $g(p)=x$ and $g(q)=1$ for $q\in Q\setminus\{p\}$.  Then $f_\lambda(p)\carrow f(p)$ because $\mathbf{d}(f(p),x)=\sup$-$\mathbf{d}(f,g)\leq\liminf_\lambda\sup$-$\mathbf{d}(f_\lambda,g)=\liminf_\lambda\mathbf{d}(f_\lambda(p),x)$.
\end{proof}

\subsection{Cauchy Nets}

We are particularly interested in the following kinds of nets.
\begin{align}
\label{Cauchy}\lim_\gamma\sup_{\gamma\prec\delta}\mathbf{d}(x_\gamma,x_\delta)=0\quad&\Leftrightarrow\quad(x_\lambda)\text{ is \emph{$\mathbf{d}$-Cauchy}}.\\
\label{dominating}\sup_\gamma\lim_{\gamma\prec\delta}\mathbf{d}(x_\gamma,x_\delta)=0\quad&\Leftrightarrow\quad(x_\lambda)\text{ is \emph{$\mathbf{d}$-dominating}}.
\end{align}

\begin{dfn}\label{YC}
$X$ is \emph{$\mathbf{d}$-complete} if every $\mathbf{d}$-Cauchy net has a $X^\circ_\circ$-limit.
\end{dfn}

As noted in \cite{Bice2016b} (1.5) and (1.6), for $\mathbf{d}$-Cauchy or $\mathbf{d}$-dominating $(x_\lambda)\subseteq X$,
\begin{align}
\label{arrowc}x_\lambda\arrowc x\quad&\Leftrightarrow\quad\mathbf{d}(x_\lambda,x)\rightarrow0.\\
\label{carrowc}x_\lambda\carrowc x\quad&\Leftrightarrow\quad x_\lambda\carrowb x\leq^\mathbf{d}x.
\end{align}
In particular, if $\mathbf{d}$ is a hemimetric then in \autoref{YC} we could replace $X^\circ_\circ$ with $X^\circ_\bullet$, showing that $\mathbf{d}$-completeness is Yoneda completeness (see \cite{Goubault2013} Definition 7.4.1).  If $\mathbf{d}$ is a metric then $X_\circ=X^\circ$ so \eqref{arrowc} shows that $\mathbf{d}$-completeness generalizes the usual notion of Cauchy completeness.

On the other hand, if we identify a preorder $\mathbin{\preceq}\subseteq X\times X$ with the function
\[\preceq(x,y)=\begin{cases}0&\text{if }x\preceq y\\ \infty&\text{otherwise}\end{cases}\]
then $\preceq$-Cauchy nets are increasing and their $X^\circ_\circ$-limits are their supremums, i.e. $\preceq$-complete $\Leftrightarrow$ directed complete.  Indeed, the primary motivation for introducing this concept in \cite{Wagner1997} and \cite{Bonsangue1998} was to unify these metric and order theoretic notions of completeness.  In \cite{Bice2016b} we took this further, showing that in more general distances spaces $\mathbf{d}$-completeness could still be characterized by combinations of metric and directed completeness.  This was based on results that are also very pertinent to the present paper, as we shall see in \autoref{NS}.








For topological spaces $Q$ and $R$ let $C(Q,R)=\{f\in R^Q:f\text{ is continuous}\}$.

\begin{thm}\label{Dini}
If $Q$ is compact, $f\in C(Q,X_\bullet)$ and $(g_\lambda)\subseteq C(Q,X^\bullet)$ is $\sup$-$\mathbf{d}$-Cauchy,
\[\lim\limits_\lambda\sup\limits_{p\in Q}\mathbf{d}(f(p),g_\lambda(p))=\sup\limits_{p\in Q}\lim\limits_\lambda\mathbf{d}(f(p),g_\lambda(p)).\]
\end{thm}

\begin{proof}
Let $s=\sup_{p\in Q}\lim_\lambda\mathbf{d}(f(p),g_\lambda(p))$, noting that the limit exists because $(g_\lambda(p))$ is $\mathbf{d}$-Cauchy \textendash\, see \cite{Bice2016b} Proposition 2.  For any $\epsilon>0$ and $\delta\in\Lambda$, define
\vspace{-4pt}
\[Q_\delta=\{p\in Q:\mathbf{d}(f(p),g_\delta(p))<s+2\epsilon\}.\]
As $(g_\lambda)$ is $\sup$-$\mathbf{d}$-Cauchy, we have $\lambda\in\Lambda$ such that $\mathbf{d}(g_\gamma(p),g_\delta(p))<\epsilon$ whenever $\lambda\prec\gamma\prec\delta$.  We claim that each $q\in Q$ is contained in the interior of $Q_\delta$, for some $\delta\succ\lambda$.  To see this, take $\delta\succ\gamma\succ\lambda$ with $\mathbf{d}(f(q),g_\gamma(q))<s+\epsilon$ and consider the sets
\begin{align*}
P&=\{p\in Q:\mathbf{d}(f(p),g_\gamma(q))<s+\epsilon\}.\\
O&=\{p\in Q:\mathbf{d}(g_\gamma(q),g_\delta(p))<\epsilon\}.
\end{align*}
As $f$ is $X_\bullet$-continuous, $P$ is an open neighbourhood of $q$.  As $g_\delta$ is $X^\bullet$-continuous, $O$ is also an open neighbourhood of $q$.  As $\mathbf{d}$ is a distance, $P\cap O\subseteq Q_\delta$.

As $Q$ is compact, we have a finite cover $Q_{\gamma_1},\ldots,Q_{\gamma_n}$ of $Q$ with $\lambda\prec\gamma_1,\ldots,\gamma_n$.  For all $\delta\succ\gamma_1,\ldots,\gamma_n$ and $p\in Q$, we have $p\in Q_{\gamma_k}$ for some $k$ and hence
\[\mathbf{d}(f(p),g_\delta(p))\leq\mathbf{d}(f(p),g_{\gamma_k}(p))+\mathbf{d}(g_{\gamma_k}(p),g_\delta(p))<s+3\epsilon.\]
So $\sup_{p\in Q}\mathbf{d}(f(p),g_\delta(p))<s+3\epsilon$ and hence $\lim_\lambda\sup_{p\in Q}\mathbf{d}(f(p),g_\lambda(p))\leq s$, as $\delta$ and $\epsilon$ were arbitrary.  The reverse inequality is immediate.
\end{proof}

Say $(g_\lambda)$ is an increasing net of continuous functions from some compact $Q$ to $\mathbb{R}$ converging pointwise to a continuous function $f$.  Thus $(g_\lambda)$ is $\mathbf{d}$-Cauchy for the quasimetric $\mathbf{d}(s,t)=|s-t|_\leq=(s-t)\vee0$ on $\mathbb{R}$.  Applying \autoref{Dini},
\[\lim\limits_\lambda\sup\limits_{p\in Q}|f(p)-g_\lambda(p)|=\lim\limits_\lambda\sup\limits_{p\in Q}\mathbf{d}(f(p),g_\lambda(p))=\sup\limits_{p\in Q}\lim\limits_\lambda\mathbf{d}(f(p),g_\lambda(p))=0,\]
i.e. $(g_\lambda)$ converges uniformly to $f$.  Thus \autoref{Dini} generalizes Dini's theorem.

The extra generality gained by not requiring $f$ to be a pointwise limit of $(g_\lambda)$ is not so important when $X=\mathbb{R}$, or even when $X$ is a lattice-ordered unital normed space, for as long as $f$ is $X^\bullet_\bullet$-continuous then we can replace each $g_\lambda$ with $(g_\lambda+s)\wedge f$, where $s=\sup_{p\in Q}\lim_\lambda\mathbf{d}(f(p),g_\lambda(p))$.  However, it is important for more general distance spaces without lattice or vector space operations.  It would also be interesting to know if \autoref{Dini} could be obtained from some purely topological version of Dini's theorem, like that given in \cite{Kupka1998}.

\subsection{Finiteness}  Define the $\mathbf{d}$-\emph{finite} elements $X^\mathsf{F}$ of $X$ by
\[X^\mathsf{F}=\{y\in X:\mathbf{d}(y,x_\lambda)\rightarrow\mathbf{d}(y,x)\text{ whenever $(x_\lambda)\subseteq X$ is $\mathbf{d}$-Cauchy and }x_\lambda\carrowc x\}.\]
This comes from \cite{Goubault2013} Definition 7.4.55 and corresponds to the usual notion of finite when $X=[0,\infty]$ and $\mathbf{d}(x,y)=|x-y|_\leq$, i.e. $[0,\infty]^\mathsf{F}=[0,\infty)$, as noted in \cite{Goubault2013} Exercise 7.4.57.  In general, $\mathbf{d}$-finiteness can also be defined from the way-below distance $\mathbf{dd}$ from \cite{KostanekWaszkiewicz2011} \S9, specifically $x\in X^\mathsf{F}\,\Leftrightarrow\,x\leq^\mathbf{dd}x$ where
\[\mathbf{dd}(x,y)=\sup\{\lim_\lambda|\mathbf{d}(x,z_\lambda)-\mathbf{d}(y,z)|_\leq:(z_\lambda)\text{ is $\mathbf{d}$-Cauchy and }z_\lambda\carrowc z\}.\]
Also, whenever $(x_\lambda)\subseteq X$ is $\mathbf{d}$-Cauchy and $x_\lambda\carrowc x$, $\mathbf{d}(y,x)\leq\lim\mathbf{d}(y,x_\lambda)$ so, in the definition of $X^\mathsf{F}$, we can replace $\mathbf{d}(y,x_\lambda)\rightarrow\mathbf{d}(y,x)$ with $\lim\mathbf{d}(y,x_\lambda)\leq\mathbf{d}(y,x)$.

The $\mathbf{d}$-finite elements are important because of the \emph{$\mathbf{d}$-finite topology} $X^{\mathsf{F}\bullet}$ they define on $X$.  This is an analog of the Scott topology on $\mathbb{R}$ (see \cite{Goubault2013} Proposition 7.4.68), at least when $X$ is $\mathbf{d}$-complete (and $\mathbf{d}$-algebraic, as defined below).


\begin{thm}\label{Zcont}
If $X$ is $\mathbf{d}$-complete then $C(Q,X^{\mathsf{F}\bullet})$ is $\sup$-$\mathbf{d}$-complete.
\end{thm}

\begin{proof}
If $(f_\lambda)\subseteq C(Q,X^{\mathsf{F}\bullet})$ is $\sup$-$\mathbf{d}$-Cauchy then, for each $p\in Q$, $(f_\lambda(p))\subseteq X$ is $\mathbf{d}$-Cauchy.  As $X$ is $\mathbf{d}$-complete, we have $f\in X^Q$ with $f_\lambda(p)\carrowc f(p)$, for all $p\in Q$, and hence $f_\lambda\carrowc f$, by \autoref{X^Yprp}.  All we need to show is that $f\in C(Q,X^{\mathsf{F}\bullet})$.

By \eqref{arrowc}, for all $\epsilon>0$ and sufficiently large $\lambda$, we have
\begin{equation}\label{C(Y,XZ)1}
\sup\!\text{-}\mathbf{d}(f_\lambda,f)<\epsilon.
\end{equation}
By the definition of $X^\mathsf{F}$, for all $p\in Q$, $c\in X^\mathsf{F}$ and sufficiently large $\lambda$, we have
\begin{equation}\label{C(Y,XZ)2}
\mathbf{d}(c,f_\lambda(p))<\mathbf{d}(c,f(p))+\epsilon.
\end{equation}
For all $q$ in a neighbourhood of $p$, we also have
\begin{align*}
\mathbf{d}(c,f_\lambda(p)) &>\mathbf{d}(c,f_\lambda(q))-\epsilon,\quad\text{as }f_\lambda\in C(Q,X^{\mathsf{F}\bullet}),\text{ so}\\
\mathbf{d}(c,f(p)) &>\mathbf{d}(c,f_\lambda(q))-2\epsilon,\quad\text{by \eqref{C(Y,XZ)2}},\\
&\geq\mathbf{d}(c,f(q))-\mathbf{d}(f_\lambda(q),f(q))-2\epsilon,\text{ by \eqref{tri}},\\
&>\mathbf{d}(c,f(q))-3\epsilon,\text{ by \eqref{C(Y,XZ)1}}.
\end{align*}
As $\epsilon>0$ and $p\in Q$ were arbitrary, $f\in C(Q,X^{\mathsf{F}\bullet})$.
\end{proof}


Note \autoref{Zcont} simultaneously generalizes the following facts.
\begin{enumerate}
\item\label{fact1} The continuous functions to a metric space are Cauchy complete.
\item\label{fact2} The lower semicontinuous functions to $[0,\infty]$ are directed complete.
\end{enumerate}
Indeed, if $\mathbf{d}$ is a metric then $X=X^\mathsf{F}$, as noted in Proposition 7.4.59, giving \eqref{fact1}.  While if $\mathbf{d}(x,y)=|x-y|_\leq$ on $[0,\infty]$ then $[0,\infty]^{\mathsf{F}\bullet}$ is the Scott topology, giving \eqref{fact2}.


For any $Y\subseteq X$, define $Y^\sigma\subseteq Y^\mathsf{m}\subseteq Y^\mathsf{D}\subseteq Y^\mathsf{C}$ (see \cite{Bice2016b} \S1 for $Y^\mathsf{D}\subseteq Y^\mathsf{C}$) by
\begin{align*}
Y^\mathsf{C}\ &=\ X^\circ_\circ\text{-limits in $X$ of $\mathbf{d}$-Cauchy nets in }Y.\\
Y^\mathsf{D}\ &=\ X^\circ_\circ\text{-limits in $X$ of $\mathbf{d}$-dominating nets in }Y.\\
Y^\mathsf{m}\ &=\ X^\circ_\circ\text{-limits in $X$ of $\leq^\mathbf{d}$-increasing nets in }Y.\\
Y^\sigma\ &=\ X^\circ_\circ\text{-limits in $X$ of $\leq^\mathbf{d}$-increasing sequences in }Y.
\end{align*}
Note the $^\mathsf{m}$ here stands for `monotone' increasing and comes from \cite{AkemannPedersen1973}.
\begin{dfn}[\cite{Goubault2013} Definition 7.4.62]
We call $X$ \emph{$\mathbf{d}$-algebraic} if $X=X^\mathsf{FC}$.
\end{dfn}

If $X=Y^\mathsf{C}$ then, to prove $Y\subseteq X^\mathsf{F}$, we need only verify finiteness for nets in $Y$.

\begin{thm}\label{CauchyCauchy}
If $X=Y^\mathsf{C}$ and $\mathbf{d}(y,x_\lambda)\rightarrow\mathbf{d}(y,x)$, whenever $y\in Y$, $(x_\lambda)\subseteq Y$ is $\mathbf{d}$-Cauchy and $x_\lambda\carrowc x\in X$, then $Y\subseteq X^\mathsf{F}$ and hence $X$ is $\mathbf{d}$-algebraic.
\end{thm}

\begin{proof}
Take $x\in X$ and $\mathbf{d}$-Cauchy $(x_\lambda)_{\lambda\in\Lambda}\subseteq X$ with $x_\lambda\carrowc x$.  As $X=Y^\mathsf{C}$, for each $\lambda$ we have $\mathbf{d}$-Cauchy $(y_\lambda^\gamma)_{\gamma\in\Gamma_\lambda}\subseteq Y$ such that $y_\lambda^\gamma\carrowc x_\lambda$.  Order
\begin{gather*}
\nabla=\{(\lambda,\gamma,\epsilon):\lambda\in\Lambda,\gamma\in\Gamma_\lambda,\epsilon>0\text{ and }s_\lambda^\gamma=\sup_{\lambda\prec\zeta}\mathbf{d}(y_\lambda^\gamma,x_\zeta)<\epsilon\}\text{ by}\\
(\lambda,\gamma,\epsilon)\prec(\zeta,\eta,\delta)\quad\Leftrightarrow\quad\lambda\prec\zeta\text{ and }\mathbf{d}(y_{\lambda}^{\gamma},y_\zeta^\eta)<\epsilon-\delta.
\end{gather*}
By \eqref{tri}, $\prec$ is transitive and we claim that $\nabla$ is also directed.  To see this, take $(\lambda,\gamma,\epsilon),(\lambda',\gamma',\epsilon')\in\nabla$.  By the definition of $\nabla$, we have positive $\delta<\epsilon-s_\lambda^\gamma,\epsilon'-s_{\lambda'}^{\gamma'}$.  As $(x_\lambda)$ is $\mathbf{d}$-Cauchy, we can take $\zeta\succ\lambda,\lambda'$ such that
\[t_\zeta=\sup_{\zeta<\xi}\mathbf{d}(x_\zeta,x_\xi)<\delta.\]
As $\mathbf{d}(y_\lambda^\gamma,x_\zeta)\leq s_\lambda^\gamma<\epsilon-\delta$, $\mathbf{d}(y_{\lambda'}^{\gamma'},x_\zeta)\leq s_{\lambda'}^{\gamma'}<\epsilon'-\delta$ and $y_\zeta^\eta\carrowc x_\zeta$, by assumption (and \eqref{arrowc}) we can take $\eta\in\Gamma_\zeta$ with
\[\mathbf{d}(y_\lambda^\gamma,y_\zeta^\eta)<\epsilon-\delta,\quad\mathbf{d}(y_{\lambda'}^{\gamma'},y_\zeta^\eta)<\epsilon'-\delta\quad\text{and}\quad\mathbf{d}(y_\zeta^\eta,x_\zeta)<\delta-t_\zeta.\]
So $\sup_{\zeta<\xi}\mathbf{d}(y_\zeta^\eta,x_\xi)\leq\mathbf{d}(y_\zeta^\eta,x_\zeta)+\sup_{\zeta<\xi}\mathbf{d}(x_\zeta,x_\xi)<\delta-t_\zeta+t_\zeta=\delta$ and thus
\[(\lambda,\gamma,\epsilon),(\lambda',\gamma',\epsilon')\prec(\zeta,\eta,\delta)\in\nabla,\]
i.e. $\nabla$ is directed.  Moreover, we can make $\delta$ above as small as we like which, as $\nabla$ is non-empty, shows that $\inf_{(\lambda,\gamma,\epsilon)\in\nabla}\epsilon=0$ and hence $(y_\lambda^\gamma)_{(\lambda,\gamma,\epsilon)\in\nabla}$ is $\mathbf{d}$-Cauchy.

For any $(\lambda,\gamma,\epsilon)\in\nabla$, the definition of $\nabla$ and the fact that $x_\lambda\arrowc x$ yields $\mathbf{d}(y_\lambda^\gamma,x)\leq\liminf\mathbf{d}(y_\lambda^\gamma,x_\zeta)\leq s_\lambda^\gamma<\epsilon$, so $y^\gamma_\lambda\arrowc x$.
Also, for $c\in X$ and $r<\mathbf{d}(x,c)$, we can take $\zeta$ as above (with $(\lambda,\gamma,\epsilon)=(\lambda',\gamma',\epsilon')$) so that we also have $r<\mathbf{d}(x_\zeta,c)$, as $x_\lambda\carrow x$, and likewise we can take $\eta\in\Gamma_\zeta$ above so that we also have $r<\mathbf{d}(y_\zeta^\eta,c)$, as $y_\zeta^\eta\carrow x_\zeta$.  Thus $\mathbf{d}(x,c)\leq\lim_{(\lambda,\gamma,\epsilon)\in\nabla}\mathbf{d}(y_\zeta^\eta,c)$ so $y_\lambda^\gamma\carrow x$ and hence $y_\lambda^\gamma\carrowc x$.



Now, for any $y\in Y$, $(\lambda,\gamma,\epsilon)\in\nabla$ and $\lambda\prec\zeta$, we have
\[\mathbf{d}(y,x_\zeta)\leq\mathbf{d}(y,y_\lambda^\gamma)+\mathbf{d}(y_\lambda^\gamma,x_\zeta)<\mathbf{d}(y,y_\lambda^\gamma)+\epsilon.\]
Thus $\lim_{\zeta\in\Lambda}\mathbf{d}(y,x_\zeta)\leq\lim_{(\lambda,\gamma,\epsilon)\in\nabla}\mathbf{d}(y,y_\lambda^\gamma)=\mathbf{d}(y,x)$, by assumption.  As $y\in Y$ and $\mathbf{d}$-Cauchy $(x_\lambda)\subseteq X$ with $x_\lambda\carrowc x$ were arbitrary, we have $Y\subseteq X^\mathsf{F}$.
\end{proof}

Before moving on, let us make one more definition.  We call $Y\subseteq X$ \emph{$\mathbf{d}$-bounded} if
\[\inf_{x\in X}\sup_{y\in Y}\mathbf{d}(x,y)<\infty.\]
Note this simultaneously generalizes the usual metric and order theoretic notions of boundedness (again identifying any preorder $\preceq$ with its characteristic function).

\section{Banach Spaces}\label{NS}

Throughout this section, we assume
\[X\textbf{ is a preordered Banach space},\]
i.e. $X$ is a real Banach space with a preorder $\leq$ that is compatible with the normed space structure of $X$.  So $X_+=\{x\in X:x\geq0\}$ satisfies
\[X_+=\overline{X_+}=\mathbb{R}_+X_+=X_++X_+,\]
where $\overline{Y}$ denotes the norm closure of $Y$, i.e. the closure in the $\mathbf{e}$-ball topology, which we denote by $X\!\bullet$, where $\mathbf{e}$ is the canonical metric defined by
\[\mathbf{e}(x,y)=||x-y||.\]
Equivalently, any such $X_+$ defines a preorder $x\leq y\ \Leftrightarrow\ y-x\in X_+$ which turns $X$ into an preordered normed space with $X_+=\{x\in X:x\geq0\}$.  We denote the unit ball by $X^1=\{x\in X:||x||\leq1\}$ and the positive unit ball by $X^1_+=X^1\cap X_+$.

Now $X$ also has a canonical half-seminorm $||\cdot||_\leq$ defined by
\[||a||_\leq=\inf_{a\leq b}||b||\]
(see \cite{RobinsonYamamuro1983b}).  This in turn yields a canonical hemimetric $\mathbf{d}$ defined by
\[\mathbf{d}(a,b)=||a-b||_\leq\]
(see \cite{Cobzac2013}) with $\mathbin{\leq}=\mathbin{\leq^\mathbf{d}}$.  The dual $X^*$ is also naturally ordered by
\[X^*_+=\{\phi\in X^*:\phi[X_+]\subseteq\mathbb{R}_+\}.\]
Defining $Q=X^{*1}_+$ (with the weak* topology), \cite{RobinsonYamamuro1983b} Theorem 2.1 yields
\[||x||_\leq=\sup_{\phi\in Q}\phi(x).\]
We now identify $X$ with its canonical image in $X^{**}$ and take this as a definition for extending $||\cdot||_\leq$ and hence $\mathbf{d}$ (and hence $\mathbin{\leq}=\mathbin{\leq^\mathbf{d}}$) to $X^{**}$.  Define
\[X^\mathsf{S}=\{x\in X^{**}:x\text{ is lower semicontinuous on }Q\}.\]
Also let $X^{1>}=\{x\in X:||x||<1\}=\mathrm{int}(X^1)$.



\begin{thm}\label{X**sc}
$X^\mathsf{S1}\subseteq X^\mathsf{1C}$ is $\mathbf{d}$-complete and $\mathbf{d}$-algebraic with $X^1\subseteq X^\mathsf{S1F}$ and
\begin{align*}
X^\mathsf{C} &=X^\mathsf{D}.\\
X^\mathsf{C} &=X^\sigma\quad\text{if $X\!\bullet$ is separable}.\\
X^\mathsf{C} &=X^\mathsf{m}\quad\text{if $X^1$ is $\leq$-bounded}.\\
\hspace{-60pt}\text{If $X^*=X^*_+-X^*_+$ then}\ \ X^\mathsf{C} &=X^\mathsf{S},\ X^\mathsf{SF}=X\ \ \text{and}\\
X^\mathsf{C} &=\overline{X^\mathsf{m}}\quad\text{if $X^{1>}$ is $\leq$-directed}.
\end{align*}
\end{thm}

\begin{proof}  First note $X^{**1}$ is $\mathbf{d}$-complete and, for $x\in X^{**}$ and $\mathbf{d}$-Cauchy $(x_\lambda)\subseteq X^{**1}$,
\begin{equation}\label{w*circcirc}
x_\lambda\carrowc x\qquad\Leftrightarrow\qquad x_\lambda(\phi)\rightarrow x(\phi),\text{ for all }\phi\in Q.
\end{equation}
For by the Banach-Alaoglu theorem, $(x_\lambda)$ has a subnet $(x_\gamma)$ with a weak*-limit \nolinebreak$x$.  In particular, $x_\gamma(\phi)\rightarrow x(\phi)$, for all $\phi\in Q$, so $x_\gamma\carrowc x$, by \autoref{X^Yprp} (with $X=\mathbb{R}$ and $\mathbf{d}(r,s)=|r-s|_\leq$ on $\mathbb{R}$).  Thus $x$ is also a $X^\circ_\circ$-limit of the original net $(x_\lambda)$, by \cite{Bice2016b} Corollary 1.  Now if $x(\phi)=y(\phi)$, for all $\phi\in Q$, then $x=^\mathbf{d}y$, i.e. $x\leq^\mathbf{d}y\leq^\mathbf{d}x$, and so $x_\lambda\carrowc y$.  While conversely, if $y$ is any other $X^\circ_\circ$-limit of $(x_\lambda)$ then $\mathbf{d}(x,y)\leq\lim\mathbf{d}(x_\lambda,y)=0$, by \eqref{arrowc}.  Likewise $\mathbf{d}(y,x)=0$ so $x=^\mathbf{d}y$, which means $x(\phi)=y(\phi)$, for all $\phi\in Q$.  This proves \eqref{w*circcirc}, and now if $(x_\lambda)\subseteq X^\mathsf{S}$ then $x\in X^\mathsf{S}$ too, by \autoref{Zcont}.  Thus $X^\mathsf{S1}$ is $\mathbf{d}$-complete.

To show $X^\mathsf{S1}\subseteq X^\mathsf{1C}$ we argue as in \cite{Pedersen1979} Lemma 3.11.2.  Take $x\in X^\mathsf{S1}$ and consider $Y\subseteq A(X^*)(=$ the weak*-continuous affine functionals on $X^*)$ defined by
\[Y=\{y\in A(X^*):y(\phi)<x(\phi),\text{ for all }\phi\in Q\}.\]

We claim that $Y$ is a directed set when we define $||z||=\sup_{\phi\in X^{*1}}z(\phi)$ and
\[y\prec z\qquad\Leftrightarrow\qquad ||z||<1-y(0)\quad\text{and}\quad y(\phi)<z(\phi),\text{ for all }\phi\in Q.\]
To see this, first note that $x\prec y\prec z$ implies $||z||<1-y(0)<1-x(0)$, as $x(0)<y(0)$, so $\prec$ is transitive.  Now consider the supergraph of $x$
\[G=\{(\phi,r)\in X^*\times\mathbb{R}:x(\phi)\leq r\},\]
and take $y,z\in Y\subseteq X^*\times\mathbb{R}$, i.e. identify $y$ and $z$ with their graphs.  Note that $y\cap(Q\times\mathbb{R})$, $z\cap(Q\times\mathbb{R})$ and $X^{*1}\times\{-m\}$, for $m=1-\max(y(0),z(0))$, are all compact convex sets disjoint from $G$, i.e. $y(\phi),z(\phi)<x(\phi)$, for all $\phi\in Q$, and $-m<x(\phi)$, for all $\phi\in X^{*1}$.  Thus their convex hull $C$ is also a compact convex set disjoint from $G$.  As $x$ is lower semicontinuous on $Q$, $G\cap(Q\times\mathbb{R})$ is closed, and hence can be separated from $C$ by a closed hyperplane $h$ (see \cite{Megginson1998} Theorem 2.2.28).  Thus $h$ is (the graph of) an affine function on $X^*$, which is also weak*-continuous on $X^{*1}$.  For if $\phi$ is a weak*-limit of $(\phi_\lambda)\subseteq X^{*1}$ then $h(\phi_\lambda)$ has a cluster point in $[-m,m]$.  For any such cluster point $r$, we must have $h(\phi)=r$, as $h$ is closed.  Thus there is only one such cluster point, i.e. $h(\phi_\lambda)\rightarrow h(\phi)$, so $h$ is weak*-continuous on $X^{*1}$.
Thus $h$ is weak*-continuous on all of $X^*$, by Krein-\v{S}mulian (see \cite{Megginson1998} Corollary 2.7.9), i.e. $h\in A(X^*)$.  Now $y(\phi),z(\phi)<h(\phi)<x(\phi)$, for all $\phi\in Q$.  Also $||h||<m$, as $h(0)<x(0)=0$ and $-m<h(\phi)$, for all $\phi\in X^{*1}$.  Thus $y,z\prec h$.

Likewise, we can separate $(0,r)$ from $G$, for all $r<0$, so $\sup_{y\in Y}y(0)=0$.  Set
\[x_y=(y-y(0))/(1\vee||y||-y(0)),\] for all $y\in Y$, and note that $(x_y)\subseteq X^1$.  When $z\prec y$
\begin{align*}
||y-x_y|| &\leq||(1\vee||y||)y-y(0)y-y+y(0)||\leq(|||y||-1|_\leq-y(0))||y||-y(0)\\
&\leq-2z(0)(1-z(0))-z(0)=2z(0)^2-3z(0).
\end{align*}
As $\sup_{z\in Y}z(0)=0$, we have $||y-x_y||\rightarrow0$.  As $(y)_{y\in Y}$ is increasing and hence $\mathbf{d}$-Cauchy, $(x_y)_{y\in Y}$ is also $\mathbf{d}$-Cauchy with the same pointwise limit on $Q$.  Whenever $r<x(\phi)$, we can separate $(\phi,r)$ from $G$, so this pointwise limit is $x$.  Thus $x_y\carrowc x$, by \eqref{w*circcirc}.  As $x\in X^\mathsf{S1}$ was arbitrary, $X^\mathsf{S1}\subseteq X^\mathsf{1C}$.

If $x\in X$ then $x$ is continuous on $Q$.  If $(y_\lambda)\subseteq X^\mathsf{S1}$ with $y_\lambda\carrowc y\in X^\mathsf{S}$ then $y$ is a pointwise limit of $(y_\lambda)$ on $Q$, by \eqref{w*circcirc}.  Thus, by \autoref{Dini} (with $X=\mathbb{R}$) $\mathbf{d}(x,y_\lambda)\rightarrow\mathbf{d}(x,y)$.  As $x$ and $(y_\lambda)$ were arbitrary, $X^1\subseteq X^\mathsf{S1F}$.

Identifying $\leq$ with its characteristic function, we have $\mathbf{d}=\mathbf{e}\circ\mathbin{\leq}$ as
\[\mathbf{d}(x,y)=||x-y||_\leq=\inf_{x-y\leq z}||z||=\inf_{z\leq y}||x-z||=\inf_{z\leq y}\mathbf{e}(x,z).\]
Thus $X^\mathsf{C}=X^\mathsf{D}$ and, if $X$ is $\mathbf{e}$-separable, $X^\mathsf{C}=X^\sigma$, by \cite{Bice2016b} Theorem 3.  If $X^1$ is $\leq$-bounded and hence $\geq$-bounded, as $X^1=-X^1$, then $x^\bullet_r$ is also $\geq$-bounded, for all $x\in X$ and $r\in\mathbb{R}$.  Then \cite{Bice2016b} Theorem 1 yields $X^\mathsf{C}=X^\mathsf{m}$.

If $X^*=X^*_+-X^*_+$ then $X$ is $r$-generated, for some $r\in\mathbb{R}$, by \cite{AsimowEllis1980} Ch 2 Theorem 1.2.  This means that $\mathbf{e}\leq r\mathbf{d}^\vee$ where $\mathbf{d}^\vee=\mathbf{d}\vee\mathbf{d}^\mathsf{op}$ and $\mathbf{d}^\mathsf{op}(x,y)=\mathbf{d}(y,x)$.  As $\mathbf{d}\leq\mathbf{e}$, this means $\mathbf{e}$ and $\mathbf{d}^\vee$ are uniformly equivalent and hence any $\mathbf{d}$-Cauchy $(x_\lambda)\subseteq X$ with $x_\lambda\carrowc x$ has an $\mathbf{e}$-bounded subnet.  Indeed, for sufficiently large $\lambda$ all $\gamma\succ\lambda$, we have $\mathbf{d}(x_\lambda,x_\gamma),\mathbf{d}(x_\gamma,x)<1$.  Thus $x\in X^\mathsf{S}$, by \eqref{w*circcirc} and \autoref{Zcont}.  Thus $X^\mathsf{C}\subseteq X^\mathsf{S}$ and hence $X^\mathsf{C}=X^\mathsf{S}$, as $X^\mathsf{S}=\mathbb{R}X^\mathsf{S1}\subseteq X^\mathsf{C}$.

Now if $x\in X^{**}\setminus X$ then $x$ is not continuous on $X^{*1}$, again by Krein-\v{S}mulian.  As $\frac{1}{r}X^{*1}$ is contained in the convex hull of $Q$ and $-Q$, $x$ is not continuous on $Q$ either.  To see this note, as $x$ is not continuous on $X^{*1}$, we have $\phi_\lambda\rightarrow\phi$ with $x(\phi_\lambda)\not\rightarrow x(\phi)$.  Taking a subnet, we may assume $x(\phi_\lambda)\rightarrow s\neq x(\phi)$.  We also have $(\psi_\lambda),(\theta_\lambda)\subseteq rQ$ with $\phi_\lambda=\psi_\lambda-\theta_\lambda$.  Taking further subnets we may assume $(\psi_\lambda)$ and $(\theta_\lambda)$ have weak*-limits $\psi$ and $\theta$ respectively and hence $\phi=\psi-\theta$.  If $x$ were continuous on $Q$ then $x(\phi_\lambda)=x(\psi_\lambda)-x(\theta_\lambda)\rightarrow x(\psi)-x(\theta)=x(\phi)\neq s$, a contradiction.

So if $x\in X^\mathsf{S}\setminus X$ then $x(\phi)+\epsilon<\lim x(\phi_\lambda)$, for some $\phi_\lambda\rightarrow\phi$ in $Q$ and $\epsilon>0$.  Take $\mathbf{d}$-Cauchy $(x_\gamma)\subseteq X^1$ with $x_\gamma\carrowc x$ so, for all sufficiently large $\gamma$, $\mathbf{d}(x_\gamma,x)<\epsilon$.  As $x_\gamma$ is continuous on $Q$, $\lim_\lambda x_\gamma(\phi_\lambda)=x_\gamma(\phi)<x(\phi)+\epsilon<\lim_\lambda x(\phi_\lambda)$ and hence $\mathbf{d}(x,x)=0<\lim_\lambda(x(\phi_\lambda)-x(\phi_\lambda))\leq\mathbf{d}(x,x_\gamma)$, so $x\notin X^\mathsf{CF}$, i.e. $X^\mathsf{CF}\subseteq X$.  As $X^1\subseteq X^\mathsf{C1F}$ and every $\mathbf{d}$-Cauchy net has an $\mathbf{e}$-bounded subnet, $X\subseteq X^\mathsf{CF}$.

Finally, if $X^{1>}$ is $\leq$-directed then so is $x^s_\bullet$, for all $x\in X$ and $s\in\mathbb{R}$.  As $\mathbf{d}^\vee$ and $\mathbf{e}$ are uniformly equivalent, \cite{Bice2016b} Theorem 2 then yields $X^\mathsf{C}=\overline{X^\mathsf{m}}$.
\end{proof}



Incidentally, the $\mathbf{d}$-completeness of $X^\mathsf{S1}$, $X^1\subseteq X^\mathsf{S1F}$, $X^\mathsf{C}=X^\mathsf{m}$ and $X^\mathsf{C}=\overline{X^\mathsf{m}}$ parts do not need $\mathbf{e}$-completeness and so apply to general preordered normed spaces.  Also, for the $X^\mathsf{C}=X^\mathsf{D}$, $X^\mathsf{C}=X^\sigma$ and $X^\mathsf{C}=X^\mathsf{m}$ parts we could replace $X^{**}$ with any other ordered normed space (or even distance space) containing $X$.

If $X$ is the self-adjoint part of a C*-algebra then $X^{1>}$ is $\leq$-directed (see the proof of \cite{Pedersen1979} Theorem 1.4.2) and $\leq$-bounded iff $X$ is unital, showing that \autoref{X**sc} generalizes \cite{Pedersen1979} Proposition 3.11.5(=\cite{AkemannPedersen1973} Theorem 2.1).
The rest of \autoref{X**sc} generalizes \cite{Brown1988} Corollary 3.25.

Now take compact Hausdorff $Q$, consider $X\!\bullet$ on $X$ and $X^\bullet$ on $X^\mathsf{S}$ and let
\begin{align*}
C&=C(Q,X\!\bullet)\subseteq X^Q.\\
B&=B(Q,X^{**})\subseteq(X^{**})^Q.\\
S&=C(Q,X^\bullet)\cap B\subseteq(X^\mathsf{S})^Q.
\end{align*}
We now extend \autoref{X**sc} to $C$ embedded in $B$.

\begin{thm}\label{CYX}
$S^1\subseteq C^\mathsf{1C}$ is $\mathbf{d}$-complete and $\mathbf{d}$-algebraic with $C^1\subseteq S^\mathsf{1F}$ and
\begin{align*}
C^\mathsf{C} &=C^\mathsf{D}.\\
C^\mathsf{C} &=C^\sigma\quad\text{if $X\!\bullet$ and $Q$ are second countable}.\\
C^\mathsf{C} &=C^\mathsf{m}\quad\text{if $X^1$ is $\leq$-bounded}.\\
\text{If $X^*=X^*_+-X^*_+$ then}\ \ C^\mathsf{C} &=S,\ S^\mathsf{F}=C\ \ \text{and}\\
C^\mathsf{C} &=\overline{C^\mathsf{m}}\quad\text{if $X^{1>}$ is $\leq$-directed}.
\end{align*}
\end{thm}

\begin{proof}
Let us identify $X^{**}$ with the constant functions in $B$ and write $\mathbf{d}$ and $\mathbf{e}$ for $\sup$-$\mathbf{d}$ and $\sup$-$\mathbf{e}$ respectively.  As $X^\mathsf{S1}$ is $\mathbf{d}$-complete and $X\subseteq X^\mathsf{SF}$, $S^1$ is also $\mathbf{d}$-complete, by \autoref{Zcont}.  Moreover, the proof of \autoref{Zcont} shows that $B^\circ_\circ$-limits of $\mathbf{d}$-Cauchy nets in $S^1$ (which are unique up to $=^\mathbf{d}$) are pointwise limits.

Now, for any $x\in X^\mathsf{S1}\subseteq X^\mathsf{1C}$, we have $\mathbf{d}$-Cauchy $(x_\lambda)\subseteq X^1$ with $x_\lambda\carrowc x$.  So for any $y\in X^1\subseteq X^\mathsf{S1F}$, we have $\mathbf{d}(y,x_\lambda)\rightarrow\mathbf{d}(y,x)$.  Thus, for any $F\in[X^1]^{<\omega}(=$ the finite subsets of $X^1$) and $\epsilon>0$, we have $\lambda$ such that
\begin{align*}
\mathbf{d}(x_\lambda,x) &<\epsilon &\text{and}&&\sup_{y\in F}\mathbf{d}(y,x_\lambda)-\mathbf{d}(y,x) &<\epsilon.\\
\intertext{Thus for any $g\in S^1$, $F\in[C^1]^{<\omega}$, $p\in Q$ and $\epsilon>0$, we have $x_p\in X^1$ such that}
\mathbf{d}(x_p,g(p)) &<\epsilon &\text{and}&&\sup_{f\in F}\mathbf{d}(f(p),x_p)-\mathbf{d}(f(p),g(p))&<\epsilon,\\
\intertext{As $g$ is $X^\bullet$-continuous and each $f\in F$ is $X\!\bullet$-continuous and hence $X_\bullet$-continuous, for all $q$ in some open $O_p\ni p$ and all $f\in F$,}
\mathbf{d}(x_p,g(q)) &<\epsilon &\text{and}&&\sup_{f\in F}\mathbf{d}(f(q),x_p)-\mathbf{d}(f(p),g(p))&<\epsilon.\\
\intertext{As $Q$ is compact, we have $p_1,\ldots,p_n\in X$ with $Q=\bigcup O_{p_k}$.  As $Q$ is also Hausdorff, we have a partition of unity $u_1,\ldots,u_n\in C(Q,[0,1])$, i.e. such that $\sum u_k=1$ and $u_k^{-1}(0,1]\subseteq O_{p_k}$, for all $k$.  Defining $h_{F,\epsilon}=\sum u_kx_{p_k}\in C^1$, we then have}
\mathbf{d}(h_{F,\epsilon},g) &<\epsilon &\text{and}&& \sup_{f\in F}\mathbf{d}(f,h_{F,\epsilon})-\mathbf{d}(f,g)&<\epsilon.
\end{align*}

We claim that $h_{F,\epsilon}\carrow g$, ordering $[C^1]^{<\omega}\times(0,\infty)$ by $\subseteq\times>$.  To see this, take $p\in Q$, $z\in X^{**}$ and $\epsilon>0$.  We have $\mathbf{d}$-Cauchy $x_\lambda\carrowc g(p)$ so, for sufficiently large $\lambda$,
\[\mathbf{d}(x_\lambda,g(p))<\epsilon\qquad\text{and}\qquad\mathbf{d}(g(p),z)<\mathbf{d}(x_\lambda,z)+\epsilon.\]
Arguing as above (with $F=\emptyset$) yields $f\in C^1$ with $f(p)=x_\lambda$ and $\mathbf{d}(f,g)<\epsilon$.  Then, whenever $f\in F$, the definition of $h_{F,\delta}$ yields $\mathbf{d}(f,h_{F,\delta})<\mathbf{d}(f,g)+\delta$ so
\begin{align*}
\mathbf{d}(g(p),z)&<\mathbf{d}(f(p),z)+\epsilon\\
&\leq\mathbf{d}(f(p),h_{F,\delta}(p))+\mathbf{d}(h_{F,\delta}(p),z)+\epsilon\\
&<\mathbf{d}(f,g)+\delta+\mathbf{d}(h_{F,\delta}(p),z)+\epsilon\\
&\leq\mathbf{d}(h_{F,\delta}(p),z)+2\epsilon+\delta.
\end{align*}
Thus $h_{F,\delta}(p)\carrow g(p)$ and hence $h_{F,\delta}\carrow g$, by \autoref{X^Yprp}.

Also, whenever $h_{F,\epsilon}\in G$, we have $\mathbf{d}(h_{F,\epsilon},h_{G,\delta})<\epsilon+\delta$, showing that $(h_{F,\epsilon})$ is $\mathbf{d}$-pre-Cauchy and hence has a $\mathbf{d}$-Cauchy subnet $(h_\lambda)$, by \cite{Bice2016b} Proposition 1.  As $\mathbf{d}(h_{F,\epsilon},g)<\epsilon$, we have $h_\lambda\arrowc g$ and hence $h_\lambda\carrowc g$.  Thus
\[S^1\subseteq C^\mathsf{1C}.\]

For any $p\in Q$ and $\mathbf{d}$-Cauchy $(g_\lambda)\subseteq C^1$ with $g_\lambda\carrowc g\in S^1$, $(g_\lambda(p))$ is $\mathbf{d}$-Cauchy and $g_\lambda(p)\carrowc g(p)\in X^\mathsf{S}$, as mentioned above.  For any $f\in C$, $f(p)\in X\subseteq X^\mathsf{SF}$ so
\[\lim_\lambda\mathbf{d}(f,g_\lambda)=\lim_\lambda\sup_{p\in Q}\mathbf{d}(f(p),g_\lambda(p))=\sup_{p\in Q}\lim_\lambda\mathbf{d}(f(p),g_\lambda(p))=\mathbf{d}(f,g),\]
by \autoref{Dini}.  Thus $C^1\subseteq S^\mathsf{1F}$, by \autoref{CauchyCauchy}.

We next claim that on $C$ we have $\mathbf{d}=\mathbf{e}\circ\mathbin{\leq}$ (i.e. the supremum half-seminorm $\sup_{p\in Q}||f(p)||_\leq$ coincides with the canonical half-seminorm $\inf_{f\leq g}\sup_{p\in Q}||g(p)||$).
As $\mathbf{d}=\mathbf{e}\circ\mathbin{\leq}$ on $X$, given $f,g\in C$, $\epsilon>0$, and $p\in Q$, we have $x_p\in X$ with
\begin{align*}
\mathbf{e}(f(p),x_p) &<\mathbf{d}(f,g)+\epsilon &\text{and}&&\mathbf{d}(x_p,g(p))&=0.\\
\intertext{As $g$ is $X^\bullet$-continuous and $f$ is $X\!\bullet$-continuous, for all $q$ in some open $O_p\ni p$,}
\mathbf{e}(f(q),x_p) &<\mathbf{d}(f,g)+\epsilon &\text{and}&&\mathbf{d}(x_p,g(q))&<\epsilon.\\
\intertext{As $Q$ is compact, we have $p_1,\ldots,p_n\in X$ with $Q=\bigcup O_{p_k}$.  As $Q$ is also Hausdorff, we have $u_1,\ldots,u_n\in C(Q,[0,1])$ with $\sum u_k=1$ and $u_k^{-1}(0,1]\subseteq O_{p_k}$, for all $k$.  Defining $h=\sum u_kx_{p_k}\in C$, we then have}
\mathbf{e}(f,h) &<\mathbf{d}(f,g)+\epsilon &\text{and}&& \mathbf{d}(h,g) &<\epsilon.
\end{align*}
So $(\mathbf{e}\circ\mathbin{\leq^\mathbf{d}_\epsilon})\leq\mathbf{d}$, for all $\epsilon>0$, where $h\leq^\mathbf{d}_\epsilon g\Leftrightarrow\mathbf{d}(h,g)<\epsilon$.  By \autoref{Zcont}, $C$ is $\mathbf{e}$-complete so, by \cite{Bice2016b} Theorem 3, we have $C^\mathsf{C}=C^\mathsf{D}$,
\[(\mathbf{e}\circ\mathbin{\leq})=\sup_{\epsilon>0}(\mathbf{e}\circ\mathbin{\leq^\mathbf{d}_\epsilon})\leq\mathbf{d}\leq(\mathbf{d}\circ\mathbf{d})\leq(\mathbf{e}\circ\mathbin{\leq}),\]
and, if $X\!\bullet$ and $Q$ are second countable so $C\bullet$ is separable, $C^\mathsf{C}=C^\sigma$.

If $X^1$ is $\leq$-bounded then $C^1$ is too so $C^\mathsf{C}=C^\mathsf{m}$.  If $X^*=X^*_+-X^*_+$ then $C^\mathsf{C}=S$ and $S^\mathsf{F}=C$ follow as in the proof of \autoref{X**sc}.  If $X^{1>}$ is $\leq$-directed then another compactness/partition of unity argument combined with \cite{Bice2016b} Proposition 6 shows that $C^{1>}$ is also $\leq$-directed so again \cite{Bice2016b} Theorem 2 yields $C^\mathsf{C}=\overline{C^\mathsf{m}}$.
\end{proof}

If $X=\mathcal{K}(H)_\mathrm{sa}=$ self-adjoint compact operators on a Hilbert space $H$ then $X^{**}=\mathcal{B}(H)_\mathrm{sa}=$ self-adjoint bounded operators on $H$ and $X^\mathsf{C}=\mathcal{K}(H)_\mathrm{sa}+\mathcal{B}(H)_+$ (see \cite{Brown1988} 5.A).  By \autoref{CYX}, $\overline{C^\mathrm{m}}$ consists precisely of the functions from $Q$ to $\mathcal{K}(H)_\mathrm{sa}+\mathcal{B}(H)_+$ that are continuous w.r.t. the topology generated by upper balls with centre in $\mathcal{K}(H)_\mathrm{sa}$.  Thus \autoref{CYX} is a generalization of \cite{Brown1988} 5.13, which yields this characterization for $Q=\mathbb{N}\cup\{\infty\}$.


Above we could actually take $B$ to be the entirety of $(X^{**})^Q$, as long as we are comfortable with the norm $||\cdot||=\mathbf{e}(0,\cdot)$ taking infinite values.  Likewise, we could embed $X$ in the algebraic dual $X^{*\sharp}$ of $X^*$ instead of $X^{**}$, i.e. including even unbounded linear functionals on $X^*$.  This might even be considered cleaner in the sense that $X^\mathsf{C}=X^\mathsf{S}$ and $X^\mathsf{SF}=X$ would apply even without $X^*=X^*_+-X^*_+$ (although $X^*=X^*_+-X^*_+$ is still required for $\mathbf{e}$ and $\mathbf{d}^\vee$ to be uniformly equivalent and hence for $X^\mathsf{C}=\overline{X^\mathsf{m}}$).

\newpage

\bibliography{maths}{}
\bibliographystyle{alphaurl}

\end{document}